\documentclass[11pt,a4paper]{amsart}
\usepackage{amssymb}
\usepackage{latexsym}
\usepackage{exscale}
\usepackage{amsfonts}
\usepackage{graphicx}
\usepackage{mathrsfs}
\usepackage{amsmath,amscd,amsthm}
\usepackage{bbm,pifont}
\usepackage{enumerate}
\usepackage{color}
\usepackage{amssymb}
\usepackage{latexsym}
\usepackage{exscale}

\allowdisplaybreaks

\numberwithin{equation}{section}
\newtheorem{theorem}{Theorem}[section]
\newtheorem{lemma}[theorem]{Lemma}
\newtheorem{corollary}[theorem]{Corollary}
\newtheorem{proposition}[theorem]{Proposition}

\newtheorem{definition}[theorem]{Definition}

\newtheorem{remark}[theorem]{Remark}

\begin{document}
\allowdisplaybreaks

\title[Boundedness of singular integrals on $\mathbb{H}^{n}$]
{Boundedness of singular integrals on the flag Hardy spaces on Heisenberg group}
\author{Guorong HU and Ji Li}

\address{Department of Mathematics, Jiangxi Normal University,
Nanchang, Jiangxi 330022, China}
\email{hugr1984@163.com}

\address{ Department of Mathematics, Macquarie University, NSW, 2109, Australia}
\email{ji.li@mq.edu.au}

\date{\today}

\begin{abstract}
We prove that the classical one-parameter convolution singular integrals on the Heisenberg group are bounded on multiparameter flag Hardy spaces, which satisfy `intermediate' dilation between the one-parameter anisotropic dilation and the product dilation on $\mathbb{C}^{n}\times \mathbb{R}$ implicitly.
\end{abstract}

\maketitle


\section{Introduction and statement of main results}
\setcounter{equation}{0}

The purpose of this note is to show that the classical one-parameter convolution singular integrals 
on the Heisenberg group are bounded on multiparameter flag Hardy spaces. Recall that the Heisenberg $\mathbb{H}^{n}$ is
the Lie group with underlying manifold $\mathbb{C}^{n} \times \mathbb{R}= \{[z,t]: z \in \mathbb{C}^{n}, t \in \mathbb{R}\}$
and multiplication law
\begin{equation*}
[z,t]\circ [z', t']=[z_{1},\cdots,z_{n}, t] \circ [z_{1}',\cdots,z_{n}', t']: =\Big[z_{1}+z_{1}',\cdots, z_{n}+z_{n}',
t+t'+2 \mbox{Im} \big(\sum_{j=1}^{n}z_{j} \bar{z}_{j}\big) \Big].
\end{equation*}
The identity of $\mathbb{H}^{n}$ is the origin and the inverse is given
by $[z, t]^{-1} = [-z,-t]$.
Hereafter we agree to identify $\mathbb{C}^{n}$  with $\mathbb{R}^{2n}$ and to use the following notation to denote the
points  of $\mathbb{C}^{n} \times \mathbb{R} \equiv \mathbb{R}^{2n+1}$:
$g=[z,t] \equiv [x,y,t] = [x_{1},\cdots, x_{n}, y_{1},\cdots, y_{n},t]
$
with $z = [z_{1},\cdots, z_{n}]$, $z_{j} =x_{j}+iy_{j}$ and $x_{j}, y_{j}, t \in \mathbb{R}$ for $j=1,\ldots,n$. Then, the
composition law $\circ$ can be explicitly written as
\begin{equation*}
g\circ g'=[x,y,t] \circ [x',y',t'] = [x+x', y+y', t+t' + 2\langle y, x'\rangle -2\langle x, y'\rangle],
\end{equation*}
where $\langle \cdot, \cdot\rangle$ denotes the usual inner product in $\mathbb{R}^{n}$

Consider the dilations
\begin{equation*}
\delta_{r} : \mathbb{H}^{n} \rightarrow \mathbb{H}^{n}, \quad \delta_{r}(g)=\delta_{r} ([z,t]) = [r z, r^{2}t].
\end{equation*}
A trivial computation shows that $\delta_{r}$ is an automorphism of $\mathbb{H}^{n}$ for every $r>0$.
Define a ``norm'' function $\rho$ on $\mathbb{H}^{n}$ by
\begin{equation*}
\rho (g) = \rho([z,t]):= \max\{|z|, |t|^{1/2}\}.
\end{equation*}
It is easy to see that $\rho(g^{-1}) = \rho(-g)  =\rho(g)$,  $\rho(\delta_{r}(g)) = r\rho(g)$,
$\rho(g) = 0$ if and only if $ g=0$,  and
$\rho(g\circ g')  \leq \gamma (\rho(g) + \rho(g')),$
where $\gamma >1$ is a constant.

The Haar measure on $\mathbb{H}^{n}$ is known to just coincide with the Lebesgue measure on $\mathbb{R}^{2n+1}$. 
For any measurable set $E\subset \mathbb{H}^{n},$ we denote by $|E|$ its (Harr) measure. The vector fields
\begin{align*}
T: =\frac{\partial}{\partial t}, \quad&X_{j}:=\frac{\partial}{\partial x_{j}} -2y_{j}\frac{\partial}{\partial t}, \quad
Y_{j}:=\frac{\partial}{\partial y_{j}} +2x_{j}\frac{\partial}{\partial t}, \quad j =1,\cdots, n,
\end{align*}
form a natural basis for the Lie algebra of left-invariant vector fields on $\mathbb{H}^{n}$.
For convenience we set $X_{n+j}:= Y_{j}$ for $j=1,2,\cdots, n$, and set $X_{2n+1} :=T$.
Denote by $\widetilde{X}_{j}$, $j =1, \cdots, 2n+1$, the right-invariant vector field which
coincides with $X_{j}$ at the origin. Let $\mathbb{N}$ be the set of all non-negative integers.
For any multi-index $I=(i_{1}, \cdots, i_{2n+1}) \in \mathbb{N}^{2n+1}$, we set
$X^{I} := X_{1}^{i_{1}} X_{2}^{i_{2}}\cdots X_{2n+1}^{i_{2n+1}}$ and $\widetilde{X}^{I}
:= \widetilde{X}_{1}^{i_{1}} \widetilde{X}_{2}^{i_{2}}\cdots \widetilde{X}_{2n+1}^{i_{2n+1}}$.
It is well known that (\cite{FS})
\begin{equation*}
X^{I}(f_1 \ast f_{2}) = f_{1} \ast (X^{I}f_{2} ), \quad \widetilde{X}^{I}(f_{1} \ast f_{2} ) =
(\widetilde{X}^{I}f _{1})\ast f_{2}, \quad (X^{I}f_{1}) \ast f_{2}
= f_{1} \ast (\widetilde{X}^{I}f_{2}),
\end{equation*}
and
\begin{equation*}
{X}^{I}\tilde{f}  = (-1)^{|I|} \widetilde{\widetilde{X}^{I}f},
\end{equation*}
where $\tilde{f}$ is given by $\tilde{f}(g) := f(g^{-1})$.
We  further set
\begin{equation*}
|I| : =i_{1} + \cdots + i_{2n+1} \quad \mbox{ and } \quad d(I) := i_{1} + \cdots +i_{2n} + 2i_{2n+1} .
\end{equation*}
Then $|I|$ is said to be the order of the differential operators $X^{I} $ and $\widetilde{X}^{I}$,
while $d(I)$ is said to be the homogeneous degree of $X^{I} $ and $\widetilde{X}^{I}$.

\begin{definition}[\cite{S}]
A function $\phi$ is called a normalized bump function on
$\mathbb{H}^{n}$ if $\phi$ is supported in the unit ball $\{g= [z,t] \in \mathbb{H}^{n}: \rho (g) \leq 1\}$
and
\begin{equation} \label{bump}
|\partial^{I}_{z,t}\phi(z,t)| \leq 1
\end{equation}
uniformly for all multi-indices
$I \in \mathbb{N}^{2n+1}$ with $|I| \leq N$,
for some fixed positive integer $N$.
\end{definition}
\begin{remark}
{\rm The condition \eqref{bump} is equivalent (module a constant) to the following one:
\begin{equation} \label{bumpp}
|X^{I}\phi(g)| \leq 1
\end{equation}
for all multi-indices $I$ with $|I| \leq N$. Indeed, this follows from the following the homogeneous property
of the ``norm'' $\rho$ and the fact that
\begin{align*}
X^{I}f(g) &= \sum_{|J| \leq |I|,\ d(J) \geq d(I)}P_{IJ} (g)(\partial_{z,t}^{J}f)(g)\\
(\partial_{z,t}^{I}f)(g) &= \sum_{|J| \leq |I|,\ d(J) \geq d(I)} Q_{IJ}(g) (X^{J}f)(g),
\end{align*}
where $P_{IJ}$, $Q_{IJ}$ are polynomials of homogeneous degree $d(J) - d(I)$ (see \cite{FS}) }.
\end{remark}

We assume that $K$ is a distribution
on $\mathbb{H}^{n}$ that agrees with a function $K(g)$, $g =[z,t] \neq [0,0]$, and satisfies the following regularity conditions:
\begin{align}\label{regularity}
|K(g)| \leq C \rho(g)^{-2n-2} , \quad  |\nabla_{z} K(g)| \leq C \rho(g)^{-2n -3},  \quad
|\frac{\partial}{\partial t} K(g)| \leq C \rho (g)^{-2n -4},
\end{align}
and the cancellation condition
\begin{equation}\label{cancelation}
|K(\phi^{r})| \leq C
\end{equation}
for all normalized bump function $\phi$ and for all $r >0$, where $\phi^{r}(g) = \phi (\delta_{r}(g))$.
It is well known that the classical one-parameter convolution singular
integral $T$ defined by $T(f) = f\ast K$ is bounded on $L^{p}$, $1 <p <\infty$, and on the classical
Hardy spaces on the Heisenberg group $H^{p}(\mathbb{H}^{n})$ for $p \in (p_{0},1]$. 
See \cite{FS} and \cite{S} for more details and proofs.

M\"{u}ller, Ricci and Stein (\cite{MRS}, \cite{MRS2}) proved that Marcinkiewicz multipliers 
are $L^{p}$ bounded for $1<p<\infty $ on the Heisenberg group $\mathbb{H}^{n}.$ This is
surprising since these multipliers are invariant under a two parameter
group of dilations on $\mathbb{C}^{n}\times \mathbb{R}$, while there is
\emph{no} two parameter group of \emph{automorphic} dilations on $\mathbb{H}
^{n}$. Moreover, they show that Marcinkiewicz multiplier can be characterized by convolution 
operator with the form $f\ast K$ where, however, $K$ is a flag kernel. At the endpoint estimates, 
it is natural to expect that Hardy space and $\mathrm{BMO}$ bounds are available. 
However, the lack of automorphic dilations underlies the failure of such
multipliers to be in general bounded on the classical Hardy space $H^{1}$ and also precludes 
a pure product Hardy space theory on the Heisenberg group.
This was the original motivation in \cite{HLS} (see also \cite{HLS2}) to develop a theory of \emph{flag} Hardy
spaces $H_{flag}^{p}$ on the Heisenberg group, $0<p\leq 1$, that is in a
sense `intermediate' between the classical Hardy spaces $H^{p}(\mathbb{H}
^{n})$ and the product Hardy spaces $H_{product}^{p}(\mathbb{C}^{n}\times \mathbb{R})$
(A. Chang and R. Fefferman (\cite{CF1}, \cite{CF2}, \cite{F1}, \cite{F2},
\cite{F3}). They show that singular integrals with flag kernels, which include the
aforementioned Marcinkiewicz multipliers, are bounded on $H_{flag}^{p}$, as
well as from $H_{flag}^{p}$ to $L^{p}$, for $0<p\leq 1$. Moreover, they construct a singular integral 
with a flag kernel on the Heisenberg group, which is not bounded on the classical Hardy spaces $H^{1}(\mathbb{H}
^{n}).$ Since, as pointed out in \cite{HLS, HLS2}, the flag Hardy space $H_{flag}^{p}(\mathbb{H}
^{n})$ is contained in the classical Hardy space $H^{p}(\mathbb{H}
^{n}),$ this counterexample implies that $H_{flag}^{1}(\mathbb{H}
^{n})\subsetneqq H^{1}(\mathbb{H}^{n}).$

\bigskip
A natural question aries: Is it possible that the classical one-parameter singular integrals on the 
Heisenberg group are bounded on flag Hardy spaces $H_{flag}^{p}(\mathbb{H}^{n})$?

\medskip
Note that the classical singular integrals on the Heisenberg group satisfy the one-parameter anisotropic dilation as mentioned above. However, the flag Hardy spaces do not satisfy such a dilation, but satisfy `intermediate' dilation between the one-parameter anisotropic dilation and the product dilation on $\mathbb{C}^{n}\times \mathbb{R}$ implicitly. We would like to point out that Nagel, Ricci and Stein \cite{NRS} introduced a class of singular integrals with flag kernels on the Euclidian space. They also pointed that singular integrals with flag kernels on the Euclidian space belong to product singular integrals, see Remark 2.1.7 and Theorem 2.1.11 in \cite{NRS}, where the characterizations in terms of the corresponding multipliers between the flag and product singular integrals are given. See also \cite{NRSW} for singular integrals with flag kernels on homogeneous groups. Recently, in \cite{Tan} it was proved that the classical Calderon-Zygmund convolution operators on the Euclidean space are bounded on the product Hardy spaces.

In this note we address this deficiency by showing that the classical one-parameter convolution singular integrals on $\mathbb{H}^{n}$ are bounded for flag Hardy spaces on $\mathbb{H}^{n}.$

Before stating the main results in this note,
we begin with recalling the Calder\'on's reproducing formula, Littlewood--Paley square function and the flag Hardy space $H_{flag}^{p}(\mathbb{H}^{n}).$
Let $\psi^{(1)} \in C_{c}^{\infty }\left( \mathbb{H}^{n}\right)$ and all arbitrarily large moments vanish and
such that the following Calder\'{o}n reproducing formula holds:%
\begin{equation*}
f=\int_{0}^{\infty }{({\psi}^{(1)} _{s})}^{\vee}\ast \psi^{(1)} _{s}\ast f\frac{ds}{s},\ \ \
\ \ f\in L^{2}\left( \mathbb{H}^{n}\right) ,
\end{equation*}%
where $\ast $ is Heisenberg convolution, $({\psi}^{(1)})^{\vee} \left( \zeta \right) =%
\overline{\psi^{(1)} \left( \zeta ^{-1}\right) }$ and $\psi^{(1)} _{s}\left( z,u\right)
=s^{-2n-2}\psi^{(1)} \left( \frac{z}{s},\frac{u}{s^{2}}\right) $ for $s>0$.
See Corollary 1 of \cite{GM} for the existence of the function $\psi^{(1)}.$

Let $\psi ^{\left( 2\right) }\in \mathcal{S}\left( \mathbb{R}%
\right) $ satisfying
\begin{equation*}
\int_{0}^{\infty }|\widehat{{\psi }^{(2)}}(t\eta )|^{2}\frac{dt}{t}=1
\end{equation*}%
for all $\eta \in \mathbb{R}\backslash \{0\}.$ Assume along with the following moment
conditions
\begin{eqnarray*}
\int\limits_{\mathbb{H}^{n}}z^{\alpha }u^{\beta }\psi ^{(1)}(z,u)dzdu &=&0,\
\ \ \ \ \left\vert \alpha \right\vert +2\beta \leq M, \\
\int\limits_{\mathbb{R}}v^{\gamma }\psi ^{(2)}(v)dv &=&0,\ \ \ \ \ \gamma
\geq 0.
\end{eqnarray*}%
Here the positive integer $M$ may be taken arbitrarily large.
Thus, we have
\begin{eqnarray}\label{CRF st}
f(z,u)=\int_0^\infty\int_0^\infty \big(\psi_{s,t} \big)^{\vee}\ast \psi_{s,t} \ast f(z,u) {ds\over s}{dt\over t},
\end{eqnarray}%
where $f\in L^2(\mathbb{H}^{n})$, $\widetilde{\psi}_{s,t}(\zeta)= \overline{\psi_{s,t}(\zeta^{-1})} $ for every $\zeta\in\mathbb{H}^n$, 
and the series converges in the $L^2(\mathbb{H}^{n})$ norm. Following \cite{MRS2}, a Littlewood--Paley \emph{component}
function $\psi $ is defined on $\mathbb{H}^{n}\simeq \mathbb{C}^{n}\times
\mathbb{R}$ by the partial convolution $\ast _{2}$ in the second
variable only:
\begin{equation*}
\psi (z,u)=\psi ^{(1)}\ast _{2}\psi ^{(2)}(z,u)=\int_{\mathbb{R}}\psi
^{(1)}(z,u-v)\psi ^{(2)}(v)dv,\ \ \ \ \ \left( z,u\right) \in \mathbb{C}%
^{n}\times \mathbb{R},
\end{equation*}%
and the function $\psi_{s,t} (z,u)$ is given by
$$\psi_{s,t} (z,u)=\psi_s^{(1)}\ast _{2}\psi_t ^{(2)}(z,u)=\int_{\mathbb{R}}\psi_s
^{(1)}(z,u-v)\psi_t ^{(2)}(v)dv. $$

We now set%
\begin{eqnarray*}
\psi^{\prime} _{\mathcal{Q}} &=&\psi _{j}^{\left( 1\right) }\text{ if }%
\mathcal{Q}\in \mathsf{Q}\left( j\right) , \\
\psi^{\prime} _{\mathcal{R}} &=&\psi _{j,k}=\psi _{j}^{\left( 1\right)
}\ast _{2}\psi _{k}^{\left( 2\right) }\text{ if }\mathcal{R}\in \mathsf{R}%
\left( j,k\right) ,
\end{eqnarray*}%
where $\mathcal{Q}\in \mathsf{Q}\left( j\right) $ are cubes and $\mathcal{R}\in \mathsf{R}\left( j,k\right) $ with $k<j$ are rectangles, and
\begin{equation*}
\mathsf{Q}\equiv \bigcup_{j\in \mathbb{Z}}\mathsf{Q}\left( j\right) ,
\end{equation*}%
and the collection of \emph{all strictly vertical} dyadic rectangles as%
\begin{equation*}
\mathsf{R}_{vert}\equiv \bigcup_{j>k}\mathsf{R}\left( j,k\right).
\end{equation*}%
The wavelet Calder\'on reproducing formula is then given by the following (Theorem 3 in \cite{HLS})
\begin{align}\label{HLS discrete formula}
f(z,u)=\sum_{\mathcal{Q}\in \mathsf{Q}}f_{\mathcal{Q}}\
\Psi_{\mathcal{Q}}( z,u) +\sum_{\mathcal{R}\in
\mathsf{R}_{vert}}f_{\mathcal{R}}\ \Psi_{\mathcal{R}}(
z,u) ,\ \ \ \ \ f\in \mathcal{M}_{flag}^{M'+\delta }(
\mathbb{H}^{n}) ,
\end{align}%
where
\begin{eqnarray*}
f_{\mathcal{Q}} &\equiv &c_{\alpha }\left\vert \mathcal{Q}\right\vert \ \psi
_{j,k}\ast f\left( z_{\mathcal{Q}},u_{\mathcal{Q}}\right) ,\ \ \ \ \ \text{%
for }\mathcal{Q}\in \mathsf{Q}\left( j\right) \text{ and }k\geq j, \\
f_{\mathcal{R}} &\equiv &c_{\alpha }\left\vert \mathcal{R}\right\vert \ \psi
_{j,k}\ast f\left( z_{\mathcal{R}},u_{\mathcal{R}}\right) ,\ \ \ \ \ \text{%
for }\mathcal{R}\in \mathsf{R}\left( j,k\right) \text{ and }k<j,
\end{eqnarray*}
the functions $\Psi_{\mathcal{Q}}$ and $\Psi_{\mathcal{R}}$
are in $\mathcal{M}_{flag}^{M^{\prime }+\delta }\left( \mathbb{H}^{n}\right)$ satisfying
$\big\Vert \Psi_{\mathcal{Q}}\big\Vert _{\mathcal{M}%
_{flag}^{M+\delta }( \mathbb{H}^{n}) } \lesssim \big\Vert \psi
_{\mathcal{Q}}^{\prime }\big\Vert _{\mathcal{M}_{flag}^{M+\delta }(
\mathbb{H}^{n}) }$ and
$\big\Vert \Psi_{\mathcal{R}}\big\Vert _{\mathcal{M}%
_{flag}^{M+\delta }( \mathbb{H}^{n}) } \lesssim \big\Vert \psi
_{\mathcal{R}}^{\prime }\big\Vert _{\mathcal{M}_{flag}^{M+\delta }\left(
\mathbb{H}^{n}\right) }$,
and the convergence of the series holds in both $L^{p}\left(
\mathbb{H}^{n}\right) $ and the Banach space $\mathcal{M}%
_{flag}^{M^{\prime }+\delta }\left( \mathbb{H}^{n}\right) $.

Based on the above reproducing formula, the \emph{wavelet} Littlewood--Paley square function is defined by
\begin{equation*}
S_{flag}(f)(z,u):=\left\{ \sum_{\mathcal{Q}\in \mathsf{Q}}\left\vert \psi^{\prime} _{%
\mathcal{Q}}\ast f\left( z_{\mathcal{Q}},u_{\mathcal{Q}}\right)
\right\vert ^{2}\chi _{\mathcal{Q}}\left( z,u\right) +\sum_{\mathcal{R}\in
\mathsf{R}_{vert}}\left\vert \psi^{\prime} _{\mathcal{R}}\ast f\left( z_{%
\mathcal{R}},u_{\mathcal{R}}\right) \right\vert ^{2}\chi _{\mathcal{R}%
}\left( z,u\right) \right\} ^{{\frac{{1}}{{2}}}},
\end{equation*}%
where $\left( z_{\mathcal{Q%
}},u_{\mathcal{Q}}\right) $ is any \emph{fixed} point in the cube $\mathcal{Q%
}$; and $\left( z_{\mathcal{R}},u_{\mathcal{R}}\right) $ is any \emph{fixed} point in
the rectangle $\mathcal{R}$.


We now recall the precise definition of the \emph{flag} Hardy spaces.
\begin{definition}[\cite{HLS, HLS2}]
\label{defflagHardy}Let $0<p<\infty $. Then for $M$ sufficiently large
depending on $n$ and $p$ we define the \emph{flag} Hardy space $%
H_{flag}^{p}\left( \mathbb{H}^{n}\right) $ on the Heisenberg group by%
\begin{equation*}
H_{flag}^{p}\left( \mathbb{H}^{n}\right) :=\left\{ f\in \mathcal{M}%
_{flag}^{M+\delta }\left( \mathbb{H}^{n}\right) ^{\prime }:S_{flag}(f)\in
L^{p}\left( \mathbb{H}^{n}\right) \right\} ,
\end{equation*}%
and for $f\in H_{flag}^{p}\left( \mathbb{H}^{n}\right) $ we set
\begin{equation}
\Vert f\Vert _{H_{flag}^{p}}:=\Vert S_{flag}(f)\Vert _{p}.  \label{H^p norm}
\end{equation}
\end{definition}
See \cite{HLS, HLS2} for more details about structures of dyadic cubes and strictly vertical rectangles, test function space $\mathcal{M}
_{flag}^{M+\delta }\left( \mathbb{H}^{n}\right)$ and its dual $\mathcal{M}_{flag}^{M+\delta }\left( \mathbb{H}^{n}\right) ^{\prime }.$

The main results in this note are the following
\begin{theorem}\label{main theorem}
Suppose that $K$ is a distribution kernel on $\mathbb{H}^{n}$ satisfying the regularity conditions \eqref{regularity} and  the cancelation condition \eqref{cancelation}. Then the operator $T$ defined by $T(f) := f\ast K$
is bounded on $H^{p}_{flag}(\mathbb{H}^{n})$
for ${{4n}\over{4n+1}}<p\leq 1$.
\end{theorem}
We remark that the lower bound ${{4n}\over{4n+1}}$ for $p$ in Theorem \ref{main theorem} can be getting smaller if the regularity and cancellation conditions on $K$ are required to be getting higher. We leave these details to the reader.

As a consequence of Theorem \ref{main theorem} and the duality of  $H^{1}_{flag}(\mathbb{H}^{n})$ with $BMO_{flag}(\mathbb{H}^{n})$ as given in \cite{HLS,HLS2}, we obtain
\begin{corollary}
Suppose that $K$ is a distribution kernel on $\mathbb{H}^{n}$ as given in Theorem 1.4.
Then the operator $T$ defined by $T(f) := f\ast K$
is bounded on $BMO_{flag}(\mathbb{H}^{n}).$
\end{corollary}
The main idea to show our results is to apply the discrete Calder\'on reproducing formula, almost orthogonal estimates associated with the flag structure and the Fefferman--Stein vector valued maximal function.


\vspace{0.2cm}

{\bf Notations:} Throughout this paper, $\mathbb{N}$ will denote the set of all nonnegative integers.
For any function $f$ on $\mathbb{H}^{n}$, we define $\tilde{f}(g)= f(g^{-1})$ and $f^{\vee}(g) =
\overline{\tilde{f}(g)} = \overline{f(g^{-1})}$, $g\in \mathbb{H}^{n}$.
If $h$ is a fixed point on $\mathbb{H}^{n}$, we define the function $f_{h}$ by
$f_{h}(g): = f(h\circ g)$, $g \in \mathbb{H}^{n}$. Finally, if $f$ is a function or distribution on
$\mathbb{H}^{n}$ and $r>0$, we set $D_{r} f (g) = r^{2n+2}f(\delta_{r}(g))$.


\section{Proof of Theorem \ref{main theorem}}
\setcounter{equation}{0}

Note that it was proved in \cite{HLS, HLS2} that $L^2(\mathbb{H}^{n})\cap H^{p}_{flag}(\mathbb{H}^{n})$ is dense in $H^{p}_{flag}(\mathbb{H}^{n}).$ To show Theorem \ref{main theorem}, by the Definition 1.3 of the flag Hardy space, it suffices to prove that there exists a constant $C$ such that for every $f\in L^2(\mathbb{H}^{n})\cap H^{p}_{flag}(\mathbb{H}^{n}),$
\begin{equation}\label{one-parameter}
\bigg\|\bigg\{ \sum_{\mathcal{Q}\in \mathsf{Q}}\left\vert \psi^{\prime} _{%
\mathcal{Q}}\ast T(f)\left( z_{\mathcal{Q}},u_{\mathcal{Q}}\right)
\right\vert ^{2}\chi _{\mathcal{Q}}\left( z,u\right)\bigg\} ^{{\frac{{1}}{{2}}}}\bigg\|_p\leq C\|f\|_{H^{p}_{flag}(\mathbb{H}^{n})}
\end{equation}
and
\begin{equation}\label{flag}
\bigg\|\bigg\{\sum_{\mathcal{R}\in
\mathsf{R}_{vert}}\left\vert \psi^{\prime}_{\mathcal{R}}\ast T(f)\left( z_{%
\mathcal{R}},u_{\mathcal{R}}\right) \right\vert ^{2}\chi _{\mathcal{R}%
}\left( z,u\right) \bigg\} ^{{\frac{{1}}{{2}}}}\bigg\|_p\leq C\|f\|_{H^{p}_{flag}(\mathbb{H}^{n})}.
\end{equation}%
To achieve the estimates in \eqref{one-parameter} and \eqref{flag}, we need the almost orthogonality estimates and a new version of discrete Calder\'on -type reproducing formula. We first give the almost orthogonality estimate as follows.
\begin{lemma}\label{almost ortho psi phi}
Suppose that $\varphi, \phi$ are functions on $\mathbb{H}^{n}$ satisfying that for all $g \in \mathbb{H}^{n},$
\begin{align*}
\int_{\mathbb{H}^{n}}\varphi(g)dg = 0, \int_{\mathbb{H}^{n}}\phi(g)dg = 0, \quad &\\
|\varphi(g)|, |\phi(g)| \leq C\frac{1}{( 1 +\rho(g))^{2n+3}}, \quad & \\
|\nabla_{z}\varphi(g)|, |\nabla_{z}\phi(g)| \leq C\frac{1}{( 1 +\rho(g))^{2n+4}}, \quad & and\\
|\frac{\partial}{\partial t}\varphi(g)|, |\frac{\partial}{\partial t}\phi(g)| \leq C\frac{1}{( 1 +\rho(g))^{2n+5}}.
\end{align*}
Then for any $\varepsilon  \in (0,1)$, there is a constant $C>0$ such that for all $j, j' \in \mathbb{Z},$
\begin{equation*}
|\varphi_{j} \ast \phi_{j'}(g)| \lesssim 2^{-|j-j'|\varepsilon} \frac{2^{-(j\wedge j')}}{(2^{-(j\wedge j')} + \rho(g))^{2n+3}}.
\end{equation*}
where $\varphi_{j}(g):= (D_{2^{j}}\varphi)(g) = 2^{j(2n+2)}\varphi(\delta_{2^{j}}(g))$.
\end{lemma}
The proof of Lemma \ref{almost ortho psi phi} is routing and we omit the details of the proof.

\begin{lemma} \label{orthogonal}
Suppose $K$ is a classical Calder\'{o}n--Zygmund kernel and
$\psi^{(1)}$ is a smooth function on $\mathbb{H}^{n}$ with support in  $B(0,1/100\gamma b)$
(where $\gamma >1$ is the constant
in the quasi-triangle inequality for the ``norm'') and $b >1$ is the constant in the stratified mean value theorem
\cite{FS}),
and $\int_{\mathbb{H}^{n}} \psi^{(1)}(g)dg =0$.  Then for any $\varepsilon \in (0,1)$, there
is a constant $C>0$ such that for any $0<\varepsilon<1$ and all $j, j' \in \mathbb{Z}$,
\begin{equation} \label{des}
|\psi_{j}^{(1)} \ast K \ast \psi_{j'}^{(1)}(g)| \lesssim 2^{-|j-j'|\varepsilon}\frac{2^{-(j \wedge j')}}
{ (2^{-(j \wedge j') } + \rho(g))^{2n+3}},
\end{equation}
where $\psi^{(1)}_{j}(g) : =(D_{2^{j}}\psi^{(1)})(g)= 2^{(2n+2)j}\psi(\delta_{2^{j}}(g))$.
\end{lemma}
\begin{proof}

We first recall that there is a constant $C$ independent of $j $ such that
\begin{equation} \label{key}
|(D_{2^{-j}}K) \ast \psi (g)| \leq C \frac{1}{(1+ \rho(g))^{2n+3}}.
\end{equation}
See \cite{S} for the detail of the proof.
Note that we also have
\begin{equation} \label{keyy}
|\psi^{(1)} \ast (D_{2^{-j}}K) (g)| \lesssim \frac{1}{(1+ \rho(g))^{2n+3}}.
\end{equation}
Indeed, this follows from \eqref{key}, the observation
$\psi^{(1)} \ast (D_{2^{-j}}K) (g) = (D_{2^{-j}} \widetilde{K}) \ast \widetilde{\psi^{(1)}} (g^{-1})$,
and the fact that $\widetilde{K}$ satisfies the same size, smoothness, and cancellation conditions to  $K$.

Now we can derive \eqref{des} from \eqref{key} and \eqref{keyy}. To see this, we write
\begin{equation*}
\begin{split}
\psi_{j}^{(1)} \ast K \ast \psi_{j'}^{(1)}
&= (D_{2^{j}}\psi^{(1)} ) \ast  K \ast  (D_{2^{j'}}\psi^{(1)} )  \\
& = \begin{cases}
D_{2^{j}}[\psi^{(1) } \ast (D_{2^{-j}}K) ] \ast (D_{2^{j'}}\psi^{(1)}) & \mbox{if  } j \geq j' ,\\
(D_{2^{j}}\psi^{(1)}) \ast D_{2^{j'}}[(D_{2^{-j'}}K) \ast \psi^{(1)} ] & \mbox{if   }  j < j'.
\end{cases}
\end{split}
\end{equation*}
Thus by Lemma \ref{almost ortho psi phi} we obtain
\begin{align*}
|\psi_{j}^{(1)} \ast K \ast \psi_{j'}^{(1)}(g) |
&\lesssim \begin{cases}
2^{-(j-j')\varepsilon} \frac{\displaystyle 2^{-j'}}{\displaystyle (2^{-j'}+ \rho(g))^{2n+3}}
& \mbox{if }   j  \geq j' \\[10pt]
2^{-(j'-j)\varepsilon} \frac{\displaystyle 2^{-j }}{\displaystyle (2^{-j}+ \rho(g))^{2n+3}} & \mbox{if  } j < j'
\end{cases} \\
& = 2^{-|j-j'|\varepsilon } \frac{\displaystyle 2^{-(j \wedge j') }}{\displaystyle (2^{-(j\wedge j')}+ \rho(g))^{2n+3}},
\end{align*}
for any $\varepsilon \in (0,1)$. The proof of Lemma \ref{orthogonal} is concluded.
\end{proof}

The key estimate is the following
\begin{lemma}\label{lemma orth}
Let $\psi^{(1)}$ be as in Lemma \ref{orthogonal} and let $\psi^{(2)} \in \mathcal{S}(\mathbb{R})$
with $\int_{\mathbb{R}} \psi^{(2)}(u)udu =0$.  Set $\psi^{(1)}_{j}(g): = 2^{j(2n+2)}\psi^{(1)}(\delta_{2^{j}}(g))$,
$\psi^{(2)}_{k}(u): = 2^{k}\psi^{(2)}(2^{k}u)$, and
$\psi_{j, k} (g) = \psi_{j,k}(z,u) := [\psi^{(1)}_{j}(z, \cdot) \ast_{\mathbb{R}} \psi_{k}^{(2)}](t) = \int_{\mathbb{R}}
\psi_{j}^{(1)}(z, t-u)\psi_{k}^{(2)}(u)du$.
Then, for $\varepsilon \in (0,1),$
\begin{align*}
&|\psi_{j,k} \ast K \ast \psi_{j', k'}(z,t)| \\
&\lesssim
\begin{cases}
2^{-|j-j'|\varepsilon}2^{-|k-k'|}\frac{\displaystyle 2^{-(j \wedge j')/2}}{\displaystyle (2^{-(j \wedge j')}
+ |z|)^{2n + \frac{1}{2}}}\frac{\displaystyle 2^{-(k \wedge k')/4}}{\displaystyle (2^{-k\wedge k'} + |t|)^{1 + \frac{1}{4}}}
&\mbox{if  } 2(j \wedge j') \geq k \wedge k', \\[10pt]
 2^{-|j-j'|\varepsilon}2^{-|k-k'|}\frac{2^{\displaystyle -(j \wedge j')/2}}{\displaystyle (2^{-(j \wedge j')}
+ |z|)^{2n+ \frac{1}{2}}}\frac{\displaystyle 2^{-(j \wedge j')/2}}{\displaystyle (2^{-(j \wedge j')}
+ \sqrt{|t|})^{2  + \frac{1}{2}}}  & \mbox{if } 2(j \wedge j') \leq k \wedge k'.
\end{cases}
\end{align*}
\end{lemma}

\begin{proof}
We write
\begin{align*}
\psi_{j,k} \ast K \ast \psi_{j', k'}  &= (\psi^{(1)}_{j} \ast_{\mathbb{R}} \psi_{k}^{(2)} ) \ast_{\mathbb{H}^{n}} K
\ast_{\mathbb{H}^{n}} (\psi^{(1)}_{j'} \ast_{\mathbb{R}} \psi^{(2)}_{k'} ) \\
& = (\psi_{j}^{(1)} \ast_{\mathbb{H}^{n} } K \ast_{\mathbb{H}^{n}} \psi^{(1)}_{j'} )
\ast_{\mathbb{R}}  (\psi^{(2)}_{k}\ast_{\mathbb{R}} \psi^{(2)}_{k'}).
\end{align*}
By almost orthogonal estimate on $\mathbb{R}$ we have
\begin{equation*}
|(\psi^{(2)}_{k}\ast_{\mathbb{R}} \psi^{(2)}_{k'}(t)| \lesssim 2^{-|k-k'|}
\frac{2^{-(k \wedge k)}}{(2^{-k\wedge k}+ |t|)^{2}}.
\end{equation*}
Combining this with \eqref{des}, we obtain
\begin{align*}
&|\psi_{j,k} \ast K \ast \psi_{j', k'} (z,t)|\\
 & \lesssim
\int_{\mathbb{R}} | (\psi_{j}^{(1)} \ast_{\mathbb{H}^{n} } K \ast_{\mathbb{H}^{n}} \psi^{(1)}_{j'} ) (z, t-u)|
| (\psi^{(2)}_{k}\ast_{\mathbb{R}} \psi^{(2)}_{k'})(u)|du \\
& \lesssim  2^{-|j-j'|\varepsilon}2^{-|k-k'|} \int_{\mathbb{R}}  \frac{2^{-(j \wedge j')}}{[2^{-(j \wedge j')}
+ (|z|^{2} + |t-u|)^{1/2}]^{2n+3}}
\frac{2^{-(k \wedge k')}}{(2^{-(k \wedge k')}+ |u|)^{2}}du\\
& \sim  2^{-|j-j'|\varepsilon}2^{-|k-k'|} \int_{\mathbb{R}}  \frac{2^{-(j \wedge j')}}{(2^{-2(j \wedge j')}
+ |z|^{2} + |t-u|)^{(n+1) + \frac{1}{2}}}
\frac{2^{-(k \wedge k')}}{(2^{-(k \wedge k')}+ |u|)^{2}}du\\
\end{align*}

{\bf Case 1: } If $2(j \wedge j)  \geq k \wedge k'$ and $|t| \geq 2^{-(k \wedge k')}$, write
\begin{align*}
&\int_{\mathbb{R}}  \frac{2^{-(j \wedge j')}}{(2^{-2(j \wedge j')}
+ |z|^{2} + |t-u|)^{(n+1) + \frac{1}{2}}}
\frac{2^{-(k \wedge k')}}{(2^{-(k \wedge k')}+ |u|)^{2}}du\\
& = \int_{|u| \leq \frac{1}{2}|t|, \mbox{ or  } |u| \geq 2t} + \int_{\frac{1}{2}|t| \leq |u| \leq 2|t|} = I + II.
\end{align*}
It is easy to see that
\begin{align*}
|I| & \lesssim \frac{2^{-(j \wedge j')}}{(2^{-2(j \wedge j')}
+ |z|^{2} + |t|)^{(n+1) + \frac{1}{2}}} \\
& \lesssim \frac{2^{-(j \wedge j')/2}}{(2^{-2(j \wedge j')}
+ |z|^{2} )^{n + \frac{1}{4}}} \frac{2^{-(j \wedge j')/2}}{|t|^{1 + \frac{1}{4}}} \\
& \lesssim \frac{2^{-(j \wedge j')/2}}{(2^{-(j \wedge j')}
+ |z|)^{2n + \frac{1}{2}}} \frac{2^{-(k \wedge k')/4}}{(2^{-k\wedge k'} + |t|)^{1 + \frac{1}{4}}}.
\end{align*}

Next, we estimate
\begin{align*}
|II| &\lesssim \frac{2^{-(k \wedge k')}}{(2^{-(k \wedge k')}+ |t|)^{2}} \int_{\mathbb{R}}  \frac{2^{-(j \wedge j')}}{(2^{-2(j \wedge j')}
+ |z|^{2} + |t-u|)^{(n+1) + \frac{1}{2}}}du \\
& \lesssim \frac{2^{-(k \wedge k')}}{(2^{-(k \wedge k')}+ |t|)^{2}} \int_{\mathbb{R}}
\frac{2^{-(j \wedge j')/2 }}{(2^{-2(j \wedge j')}
+ |z|^{2} )^{n  + \frac{1}{4}}}
\frac{ 2^{-(j \wedge j')/2  }}{(2^{-2(j \wedge j')}
+  |t-u|)^{1 +\frac{1}{4}} }du \\
& \lesssim \frac{2^{-(j \wedge j')/2 }}{(2^{-(j \wedge j')}
+ |z|)^ {2n  + \frac{1}{2}}} \frac{2^{-(k \wedge k')/4}}{(2^{-(k \wedge k')}+ |t|)^{1+\frac{1}{4}}}.
\end{align*}

{\bf Case 2:} If $2 (j \wedge j') \geq k \wedge k'$ and $|t | \leq 2^{-(k \wedge k')}$, then
\begin{align*}
 &\int_{\mathbb{R}}  \frac{2^{-(j \wedge j')}}{(2^{-2(j \wedge j')}
+ |z|^{2} + |t-u|)^{(n+1) + \frac{1}{2}}}
\frac{2^{-(k \wedge k')}}{(2^{-(k \wedge k')}+ |u|)^{2}}du \\
& \lesssim \frac{1}{2^{-(k \wedge k')}} \int_{\mathbb{R}}  \frac{2^{-(j \wedge j')}}{(2^{-2(j \wedge j')}
+ |z|^{2} + |t-u|)^{(n+1) + \frac{1}{2}}}du \\
& \lesssim \frac{2^{-(k \wedge k')/4}}{(2^{-(k \wedge k')}+|t|)^{1+\frac{1}{4}}} \frac{2^{-(j \wedge j')/2 }}{(2^{-(j \wedge j')}
+ |z|)^ {2n  + \frac{1}{2}}}.
\end{align*}

{\bf Case 3:} We now consider the case $2(j \wedge j') \leq k\wedge k'$
and $|t| \leq 2^{-2(j \wedge j')}$. Then
\begin{align*}
&\int_{\mathbb{R}}  \frac{2^{-(j \wedge j')}}{(2^{-2(j \wedge j')}
+ |z|^{2} + |t-u|)^{(n+1) + \frac{1}{2}}}
\frac{2^{-(k \wedge k')}}{(2^{-(k \wedge k')}+ |u|)^{2}}du \\
&\lesssim \frac{2^{-(j \wedge j')}}{(2^{-2(j \wedge j')}
+ |z|^{2})^{(n+1) + \frac{1}{2}}}  \\
& \lesssim  \frac{2^{-(j \wedge j')/2}}{(2^{-2(j \wedge j')}
+ |z|^{2} )^{n+ \frac{1}{4}}}  \frac{2^{-(j \wedge j')/2}}{(2^{-2(j \wedge j')}
+ |t|)^{1+ \frac{1}{4}}} \\
&\sim  \frac{2^{-(j \wedge j')/2}}{(2^{-(j \wedge j')}
+ |z| )^{2n+ \frac{1}{2}}}  \frac{2^{-(j \wedge j')/2}}{(2^{-(j \wedge j')}
+ \sqrt{|t|})^{2+ \frac{1}{2}}}.
\end{align*}

{\bf Case 4:}  If $2(j \wedge j') \leq k\wedge k'$ and $|t | \geq 2^{-2(j \wedge j')}$, write
\begin{align*}
&\int_{\mathbb{R}}  \frac{2^{-(j \wedge j')}}{(2^{-2(j \wedge j')}
+ |z|^{2} + |t-u|)^{(n+1) + \frac{1}{2}}}
\frac{2^{-(k \wedge k')}}{(2^{-(k \wedge k')}+ |u|)^{2}}du\\
& = \int_{|u| \leq \frac{1}{2}|t|, \mbox{ or  } |u| \geq 2t} + \int_{\frac{1}{2}|t| \leq |u| \leq 2|t|} = I + II.
\end{align*}
It is easy to see that
\begin{align*}
|I| \lesssim &  \frac{2^{-(j \wedge j')}}{(2^{-2(j \wedge j')}
+ |z|^{2} + |t|)^{(n+1) + \frac{1}{2}}} \\
& \lesssim  \frac{2^{-(j \wedge j')/2}}{(2^{-2(j \wedge j')}
+ |z|^{2})^{n+ \frac{1}{4}}}  \frac{2^{-(j \wedge j')/2}}{(2^{-2(j \wedge j')}
+ |t|)^{1  + \frac{1}{4}}} \\
& \sim  \frac{2^{-(j \wedge j')/2}}{(2^{-(j \wedge j')}
+ |z|)^{2n+ \frac{1}{2}}}  \frac{2^{-(j \wedge j')/2}}{(2^{-(j \wedge j')}
+ \sqrt{|t|})^{2  + \frac{1}{2}}} .
\end{align*}

To estimate $II$, we have
\begin{align*}
|II| &\lesssim  \frac{2^{-(k \wedge k')}}{(2^{-(k \wedge k')}+ |t|)^{2}} \int_{\mathbb{R}}  \frac{2^{-(j \wedge j')}}{(2^{-2(j \wedge j')}
+ |z|^{2} + |t-u|)^{(n+1) + \frac{1}{2}}}du\\
& \lesssim \frac{2^{-(k \wedge k')}}{(2^{-(k \wedge k')}+ |t|)^{2}} \int_{\mathbb{R}}
\frac{2^{-(j \wedge j')/2 }}{(2^{-2(j \wedge j')}
+ |z|^{2} )^{n  + \frac{1}{4}}}
\frac{ 2^{-(j \wedge j')/2  }}{(2^{-2(j \wedge j')}
+  |t-u|)^{1 +\frac{1}{4}} }du \\
& \lesssim \frac{2^{-(j \wedge j')/2 }}{(2^{-2(j \wedge j')}
+ |z|^{2} )^{n  + \frac{1}{4}}} \frac{2^{-2(j \wedge j')}}{(2^{-2(j \wedge j')}+ |t|)^{2}} \\
& \sim \frac{2^{-(j \wedge j')/2 }}{(2^{-(j \wedge j')}
+ |z|)^{2n  + \frac{1}{2}}} \frac{2^{-(j \wedge j')/2}}{(2^{-(j \wedge j')}
+ \sqrt{|t|})^{2  + \frac{1}{2}}}.
\end{align*}

This finishes the proof.
\end{proof}

Now we prove the following new version of discrete Calderon's reproducing formula.

\begin{theorem}\label{Calderon new}
Suppose $0<p\leq 1$. For any given $f\in L^2(\mathbb{H}^n)\cap H^p_{flag}(\mathbb{H}^n)$, there exists $h\in L^2(\mathbb{H}^n)\cap H^p_{flag}(\mathbb{H}^n)$ such that, for a sufficiently large integer $N \in\mathbb{N}$,
\begin{equation}\label{eq:1}
f(z,u)=\sum_{j,k\in \mathbb{Z}}\sum_{\substack{R=I\times J,\\ \ell(I)=2^{-j-N},\\\ell(J)=2^{-j-N}+2^{-k-N} }}|R| \widetilde{\psi}_{j,k}((z,u)\circ(z_I,u_J)^{-1}) (\psi_{j,k}*h)(z_I,u_J),
\end{equation}
where the series converges in $L^2(\mathbb{H}^n)$ and $z_I,u_J$ are any fixed points in $I, J$, respectively. Moreover,
\begin{equation}\label{eq:2}
    \|f\|_{H^p_{flag}(\mathbb{H}^n)}\approx\|h\|_{H^p_{flag}(\mathbb{H}^n)},\ \  \|f\|_{L^2(\mathbb{H}^n)}\approx\|h\|_2.
\end{equation}
\end{theorem}

\begin{proof}
Following \cite{HLS}(see also \cite{HLS2}) and beginning with the Calder\'on reproducing 
formula in \eqref{CRF st} that holds for $f\in L^2(\mathbb{H}^n)$ and converges in $L^2(\mathbb{H}^n),$ for any given $\alpha >0,$ we discretize
\eqref{CRF st} as follows:%

\begin{eqnarray*}
f\left( z,u\right) &=&\int_{0}^{\infty }\int_{0}^{\infty }\widetilde{\psi}%
_{s,t}\ast _{\mathbb{H}^{n}}\psi _{s,t}\ast _{\mathbb{H}^{n}}f\left(
z,u\right) \frac{ds}{s}\frac{dt}{t} \\
&=&\sum_{j,k\in \mathbb{Z}}\int_{2^{-\alpha \left( j+1\right) }}^{2^{-\alpha
j}}\int_{2^{-2\alpha \left( k+1\right) }}^{2^{-2\alpha k}}\widetilde{\psi}%
_{s,t}\ast \psi _{s,t}\ast f\left( z,u\right) \frac{dt}{t}\frac{ds}{s} \\
&=&c_{\alpha }\sum_{j\leq k}\widetilde{\psi}_{j,k}\ast \psi _{j,k}\ast f\left(
z,u\right) +c_{\alpha }\sum_{j>k}\widetilde{\psi}_{j,k}\ast \psi _{j,k}\ast
f\left( z,u\right) \\
&&+\sum_{j,k\in \mathbb{Z}}\int_{2^{-\alpha \left( j+1\right) }}^{2^{-\alpha
j}}\int_{2^{-2\alpha \left( k+1\right) }}^{2^{-2\alpha k}}\left\{ \widetilde{\psi%
}_{s,t}\ast \psi _{s,t}-\widetilde{\psi}_{j,k}\ast \psi _{j,k}\right\} \ast
f\left( z,u\right) \frac{dt}{t}\frac{ds}{s} \\
&=:&T_{\alpha }^{\left( 1\right) }f\left( z,u\right) +T_{\alpha }^{\left(
2\right) }f\left( z,u\right) +R_{\alpha }f\left( z,u\right) ,
\end{eqnarray*}%
where%
\begin{eqnarray*}
\psi _{j,k} &=&\psi _{2^{-\alpha j},2^{-2\alpha k}}, \\
c_{\alpha } &=&\int_{2^{-\alpha \left( j+1\right) }}^{2^{-\alpha
j}}\int_{2^{-2\alpha \left( k+1\right) }}^{2^{-2\alpha k}}\frac{dt}{t}\frac{%
ds}{s}=\ln \frac{2^{-\alpha j}}{2^{-\alpha \left( j+1\right) }}\ln \frac{%
2^{-2\alpha k}}{2^{-2\alpha \left( k+1\right) }}=2\left( \alpha \ln 2\right)
^{2}.
\end{eqnarray*}

We further discretize the terms $T_{\alpha }^{\left( 1\right) }f\left(
z,u\right) $ and $T_{\alpha }^{\left( 2\right) }f\left( z,u\right) $ in
different ways, exploiting the one-parameter structure of the Heisenberg
group for $T_{\alpha }^{\left( 1\right) }$, and exploiting the implicit
product structure for $T_{\alpha }^{\left( 2\right) }$. More precisely,
\begin{eqnarray*}
T_{\alpha }^{\left( 1\right) }f\left( z,u\right) &=&\sum_{j\leq k}\sum_{%
\mathcal{Q}\in \mathsf{Q}\left( j\right) }f_{\mathcal{Q}}\psi _{\mathcal{Q}%
}\left( z,u\right) +R_{\alpha ,N}^{\left( 1\right) }f\left( z,u\right) , \\
T_{\alpha }^{\left( 2\right) }f\left( z,u\right) &=&\sum_{j>k}\sum_{\mathcal{%
R}\in \mathsf{R}\left( j,k\right) }f_{\mathcal{R}}\psi _{\mathcal{R}}\left(
z,u\right) +R_{\alpha ,N}^{\left( 2\right) }f\left( z,u\right) ,
\end{eqnarray*}%
where%
\begin{eqnarray*}
f_{\mathcal{Q}} &\equiv &c_{\alpha }\left\vert \mathcal{Q}\right\vert \ \psi
_{j,k}\ast f\left( z_{\mathcal{Q}},u_{\mathcal{Q}}\right) ,\ \ \ \ \ \text{%
for }\mathcal{Q}\in \mathsf{Q}\left( j\right) \text{ and }k\geq j, \\
f_{\mathcal{R}} &\equiv &c_{\alpha }\left\vert \mathcal{R}\right\vert \ \psi
_{j,k}\ast f\left( z_{\mathcal{R}},u_{\mathcal{R}}\right) ,\ \ \ \ \ \text{%
for }\mathcal{R}\in \mathsf{R}\left( j,k\right) \text{ and }k<j, \\
\psi _{\mathcal{Q}}\left( z,u\right) &=&\frac{1}{\left\vert \mathcal{Q}%
\right\vert }\int_{\mathcal{Q}}\widetilde{\psi}_{j,k}\left( \left( z,u\right)
\circ \left( z^{\prime },u^{\prime }\right) ^{-1}\right) dz^{\prime
}du^{\prime },\ \ \ \ \ \text{for }\mathcal{Q}\in \mathsf{Q}\left( j\right)
\text{ and }k\geq j, \\
\psi _{\mathcal{R}}\left( z,u\right) &=&\frac{1}{\left\vert \mathcal{R}%
\right\vert }\int_{\mathcal{R}}\widetilde{\psi}_{j,k}\left( \left( z,u\right)
\circ \left( z^{\prime },u^{\prime }\right) ^{-1}\right) dz^{\prime
}du^{\prime },\ \ \ \ \ \text{for }\mathcal{R}\in \mathsf{R}\left(
j,k\right) \text{ and }k<j.
\end{eqnarray*}%
and%
\begin{eqnarray*}
R_{\alpha ,N}^{\left( 1\right) }f\left( z,u\right) &=&c_{\alpha }\sum_{j\leq
k}\sum_{\mathcal{Q}\in \mathsf{Q}\left( j\right) }\int_{\mathcal{Q}}\widetilde{%
\psi}_{j,k}\left( \left( z,u\right) \circ \left( z^{\prime },u^{\prime
}\right) ^{-1}\right) \\
&&\times \left[ \psi _{j,k}\ast f\left( z^{\prime },u^{\prime }\right) -\psi
_{j,k}\ast f\left( z_{\mathcal{Q}},u_{\mathcal{Q}}\right) \right] dz^{\prime
}du^{\prime }, \\
R_{\alpha ,N}^{\left( 2\right) }f\left( z,u\right) &=&c_{\alpha
}\sum_{j>k}\sum_{\mathcal{R}\in \mathsf{R}\left( j,k\right) }\int_{\mathcal{R%
}}\widetilde{\psi}_{j,k}\left( \left( z,u\right) \circ \left( z^{\prime
},u^{\prime }\right) ^{-1}\right) \\
&&\times \left[ \psi _{j,k}\ast f\left( z^{\prime },u^{\prime }\right) -\psi
_{j,k}\ast f\left( z_{\mathcal{R}},u_{\mathcal{R}}\right) \right] dz^{\prime
}du^{\prime }.
\end{eqnarray*}

Altogether we have%
\begin{eqnarray}
f\left( z,u\right) &=&\sum_{j\in \mathbb{Z}}\sum_{\mathcal{Q}\in \mathsf{Q}%
\left( j\right) }f_{\mathcal{Q}}\psi _{\mathcal{Q}}\left( z,u\right)
+\sum_{j>k}\sum_{\mathcal{R}\in \mathsf{R}\left( j,k\right) }f_{\mathcal{R}%
}\psi _{\mathcal{R}}\left( z,u\right)  \label{altogether} \\
&&+\left\{ R_{\alpha }f\left( z,u\right) +R_{\alpha ,N}^{\left( 1\right)
}f\left( z,u\right) +R_{\alpha ,N}^{\left( 2\right) }f\left( z,u\right)
\right\} .  \notag
\end{eqnarray}%
Recall that we denote by $\mathsf{Q}\equiv \bigcup_{j\in \mathbb{Z}}\mathsf{Q%
}\left( j\right) $ the collection of \emph{all} dyadic cubes, and by $%
\mathsf{R}_{vert}\equiv \bigcup_{j>k}\mathsf{R}\left( j,k\right) $ the
collection of \emph{all strictly vertical} dyadic rectangles. Finally, we can
rewrite the right-hand side of the equality (\ref{altogether}) as%
\begin{align}
f\left( z,u\right) &=\bigg(\sum_{\mathcal{Q}\in \mathsf{Q}}f_{\mathcal{Q}}\psi _{%
\mathcal{Q}}\left( z,u\right) +\sum_{\mathcal{R}\in \mathsf{R}_{vert}}f_{%
\mathcal{R}}\psi _{\mathcal{R}}\left( z,u\right)\bigg) +\left\{ R_{\alpha
}+R_{\alpha ,N}^{\left( 1\right) }+R_{\alpha ,N}^{\left( 2\right) }\right\}
(f)\left( z,u\right)   \label{altogether'}\\
&=: T_N(f)+R_N(f),\nonumber
\end{align}%
where the series converge in the norm of $L^2(\mathbb{H}^n).$

It was proved in \cite{HLS, HLS2} that
\begin{eqnarray*}
&&\left\Vert R_{\alpha }f\right\Vert _{L^{p}\left( \mathbb{H}^{n}\right)
}+\big\Vert R_{\alpha ,N}^{\left( 1\right) }f\big\Vert _{L^{p}\left(
\mathbb{H}^{n}\right) }+\big\Vert R_{\alpha ,N}^{\left( 2\right)
}f\big\Vert _{L^{p}\left( \mathbb{H}^{n}\right) }  \label{error op bounds}
 \leq C2^{-N}\left\Vert f\right\Vert _{L^{p}\left(
\mathbb{H}^{n}\right) }\\[3pt]
&&\ \ \ \  \textup{for all } f\in L^{p}\left( \mathbb{H}^{n}\right),
1<p<\infty ,  \notag \\[7pt]
&&\left\Vert R_{\alpha }f\right\Vert _{\mathcal{M}_{flag}^{M^{\prime
}+\delta }\left( \mathbb{H}^{n}\right) }+\big\Vert R_{\alpha ,N}^{(
1) }f\big\Vert _{\mathcal{M}_{flag}^{M^{\prime }+\delta }\left(
\mathbb{H}^{n}\right) }+\big\Vert R_{\alpha ,N}^{\left( 2\right)
}f\big\Vert _{\mathcal{M}_{flag}^{M^{\prime }+\delta }\left( \mathbb{H}%
^{n}\right) }   \leq C2^{-N}\left\Vert f\right\Vert _{\mathcal{M}%
_{flag}^{M^{\prime }+\delta }\left( \mathbb{H}^{n}\right) }\notag \\
&&\ \ \ \ \textup{for all } f\in
\mathcal{M}_{flag}^{M^{\prime }+\delta }\left( \mathbb{H}^{n}\right) .
\notag
\end{eqnarray*}
Thus, we have
$$\Big\|\left\{ R_{\alpha
}+R_{\alpha ,N}^{\left( 1\right) }+R_{\alpha ,N}^{\left( 2\right) }\right\}(f)\Big\|_{L^2(\mathbb{H}^{n})}\leq C2^{-N}\|f\|_{L^2(\mathbb{H}^{n})}.$$
Next
we claim that
\begin{eqnarray}\label{claim 3}
\Big\|\left\{ R_{\alpha
}+R_{\alpha ,N}^{\left( 1\right) }+R_{\alpha ,N}^{\left( 2\right) }\right\}(f)\Big\|_{H^p_{flag}(\mathbb{H}^{n})}\leq C2^{-N}\|f\|_{H^p_{flag}(\mathbb{H}^{n})}.
\end{eqnarray}

Indeed, the above claim follows from the following general result:
\begin{proposition}\label{boundedness on molecular}
If $T$ is a bounded operator on $L^2(\mathbb{H}^n)$ and molecular space $\mathcal{M}_{flag}^{M^{\prime }+\delta } (\mathbb{H}^n),$ then $T$ is bounded on $H^p_{flag}.$ Moreover,
$$\|T(f)\|_{H^p_{flag}}\leq C\Big(\|T\|_{2,2}+\|T\|_{\mathcal{M}_{flag}^{M^{\prime }+\delta },\mathcal{M}_{flag}^{M^{\prime }+\delta }}\Big)\|f\|_{H^p_{flag}},$$
where we denote $\|T\|_{2,2}$ for the operator norm of $T$ on $L^2(\mathbb{H}^{n})$ and $\|T\|_{\mathcal{M}_{flag}^{M^{\prime }+\delta } , \mathcal{M}_{flag}^{M^{\prime }+\delta } }$ for the operator norm on the molecular space $\mathcal{M}_{flag}^{M^{\prime }+\delta }.$
\end{proposition}
Proposition \ref{boundedness on molecular} follows from the discrete Caldero\'n's reproducing formula \eqref{HLS discrete formula}  (Theorem 3 in \cite{HLS}) and the almost orthogonality estimates (Lemma 6 in \cite{HLS}).
We only give an outline of the proof.

Suppose $f\in L^2(\mathbb{H}^{n})\cap H^p_{flag}(\mathbb{H}^{n})$. By \eqref{HLS discrete formula}, it follows that
$$T(f)\left( z,u\right) =\sum_{\mathcal{Q}\in \mathsf{Q}%
}f_{\mathcal{Q}}T({\Psi} _{\mathcal{Q}})\left( z,u\right)
+\sum_{\mathcal{R}\in {\mathsf{R}_{vert}}}f_{\mathcal{R}%
}T({\Psi} _{\mathcal{R}})\left( z,u\right).$$
Thus,
\begin{eqnarray*}
\Vert Tf\Vert _{H_{flag}^{p}}^p&=&\Vert S_{flag}(Tf)\Vert _{p}^p\\
&\leq&\bigg\|\bigg\{ \sum_{\mathcal{Q}\in \mathsf{Q}}\left\vert \psi^{\prime} _{%
\mathcal{Q}}\ast Tf\left( z_{\mathcal{Q}},u_{\mathcal{Q}}\right)
\right\vert ^{2}\chi _{\mathcal{Q}}\left( z,u\right) \bigg\} ^{{\frac{{1}}{{2}}}} \bigg\|_p^p \\
&&\quad+  \bigg\|\bigg\{\sum_{\mathcal{R}\in
\mathsf{R}_{vert}}\left\vert \psi^{\prime} _{\mathcal{R}}\ast Tf\left( z_{%
\mathcal{R}},u_{\mathcal{R}}\right) \right\vert ^{2}\chi _{\mathcal{R}%
}\left( z,u\right) \bigg\} ^{{\frac{{1}}{{2}}}}\bigg\|_p^p\\
&\leq&\bigg\|\bigg\{ \sum_{\mathcal{Q}\in \mathsf{Q}}\Big\vert \psi^{\prime} _{%
\mathcal{Q}}\ast \sum_{\mathcal{Q}'\in \mathsf{Q}%
}f_{\mathcal{Q}'}T({\Psi} _{\mathcal{Q}'})\left( z_{\mathcal{Q}},u_{\mathcal{Q}}\right)
\Big\vert ^{2}\chi _{\mathcal{Q}}\left( z,u\right) \bigg\} ^{{\frac{{1}}{{2}}}} \bigg\|_p^p\\
&&\hskip.5cm +  \bigg\|\bigg\{\sum_{\mathcal{R}\in
\mathsf{R}_{vert}}\Big\vert \psi^{\prime} _{\mathcal{R}}\ast \sum_{\mathcal{Q}'\in \mathsf{Q}%
}f_{\mathcal{Q}'}T({\Psi} _{\mathcal{Q}'})\left( z_{%
\mathcal{R}},u_{\mathcal{R}}\right) \Big\vert ^{2}\chi _{\mathcal{R}%
}\left( z,u\right) \bigg\} ^{{\frac{{1}}{{2}}}}\bigg\|_p^p\\
&&\hskip.6cm+\bigg\|\bigg\{ \sum_{\mathcal{Q}\in \mathsf{Q}}\Big\vert \psi^{\prime} _{%
\mathcal{Q}}\ast \sum_{\mathcal{R}'\in {\mathsf{R}_{vert}}}f_{\mathcal{R}'%
}T({\Psi} _{\mathcal{R}'})\left( z_{\mathcal{Q}},u_{\mathcal{Q}}\right)
\Big\vert ^{2}\chi _{\mathcal{Q}}\left( z,u\right) \bigg\} ^{{\frac{{1}}{{2}}}} \bigg\|_p^p\\
&&\hskip.7cm +  \bigg\|\bigg\{\sum_{\mathcal{R}\in
\mathsf{R}_{vert}}\Big\vert \psi^{\prime} _{\mathcal{R}}\ast \sum_{\mathcal{R}'\in {\mathsf{R}_{vert}}}f_{\mathcal{R}'%
}T({\Psi} _{\mathcal{R}'})\left( z_{%
\mathcal{R}},u_{\mathcal{R}}\right) \Big\vert ^{2}\chi _{\mathcal{R}%
}\left( z,u\right) \bigg\} ^{{\frac{{1}}{{2}}}}\bigg\|_p^p\\
&=:& A_1+A_2+A_3+A_4.
\end{eqnarray*}
To estimate the term $A_1,$ note that
$${\Psi} _{\mathcal{Q}'}\left( z,u\right) =\frac{1}{\left\vert \mathcal{Q}'%
\right\vert }\int_{\mathcal{Q}'}\widetilde{\psi}_{j',k'}\left( \left( z,u\right)
\circ \left( z^{\prime },u^{\prime }\right) ^{-1}\right) dz^{\prime
}du^{\prime }.$$
We have
\begin{eqnarray*}
A_1=\bigg\|\bigg\{ \sum_{\mathcal{Q}\in \mathsf{Q}}\bigg\vert \sum_{\mathcal{Q}'\in \mathsf{Q}%
}f_{\mathcal{Q}'}\ \frac{1}{\left\vert \mathcal{Q}'%
\right\vert }\int_{\mathcal{Q}'} \psi^{\prime} _{%
\mathcal{Q}}\ast T \widetilde{\psi}_{j',k'}\left( \left( z_{\mathcal{Q}},u_{\mathcal{Q}}\right)
\circ \left( z^{\prime },u^{\prime }\right) ^{-1}\right) dz^{\prime
}du^{\prime }
\bigg\vert ^{2}\chi _{\mathcal{Q}}\left( z,u\right) \bigg\} ^{{\frac{{1}}{{2}}}} \bigg\|_p^p\\
\end{eqnarray*}
Since $T$ is bounded on the molecular space $\mathcal{M}_{flag}^{M^{\prime }+\delta } (\mathbb{H}^n),$ we obtain that
$T \widetilde{\psi}_{j',k'}$ satisfies the same conditions as $\widetilde{\psi}_{j',k'}$ does with an extra constant
$\|T\|_{\mathcal{M}_{flag}^{M^{\prime }+\delta } (\mathbb{H}^n),\mathcal{M}_{flag}^{M^{\prime }+\delta } (\mathbb{H}^n)}$.
Thus, by Lemma 6 in \cite{HLS2}, we have
\begin{align*}
&\Big|\psi^{\prime} _{%
\mathcal{Q}}\ast T \widetilde{\psi}_{j',k'}\left( \left( z_{\mathcal{Q}},u_{\mathcal{Q}}\right)
\circ \left( z^{\prime },u^{\prime }\right) ^{-1}\right)\Big| \\
&\lesssim
\begin{cases}
\|T\|_{\mathcal{M}_{flag}^{M^{\prime }+\delta },\mathcal{M}_{flag}^{M^{\prime }+\delta }}2^{-|j-j'|\varepsilon}2^{-|k-k'|}  \frac{\displaystyle 2^{-(j \wedge j')/2}}{\displaystyle (2^{-(j \wedge j')}
+ |z_{\mathcal{Q}}-z'|)^{2n + \frac{1}{2}}}\frac{\displaystyle 2^{-(k \wedge k')/4}}{\displaystyle (2^{-k\wedge k'} + |u_{\mathcal{Q}}-u'|)^{1 + \frac{1}{4}}}\\[13pt]
&\hskip-4cm\mbox{if  } 2(j \wedge j') \geq k \wedge k'; \\[12pt]
\|T\|_{\mathcal{M}_{flag}^{M^{\prime }+\delta },\mathcal{M}_{flag}^{M^{\prime }+\delta }} 2^{-|j-j'|\varepsilon}2^{-|k-k'|}\frac{\displaystyle 2^{-(j \wedge j')/2}}{\displaystyle (2^{-(j \wedge j')}
+ |z_{\mathcal{Q}}-z'|)^{2n+ \frac{1}{2}}}\frac{\displaystyle 2^{-(j \wedge j')/2}}{\displaystyle (2^{-(j \wedge j')}
+ \sqrt{|u_{\mathcal{Q}}-u'|})^{2  + \frac{1}{2}}}\\[13pt]
  &\hskip-4cm \mbox{if } 2(j \wedge j') \leq k \wedge k'.
\end{cases}
\end{align*}

Then following the same steps as in the proof of Plancherel--P\'olya inequalities for the Hardy spaces $H^p_{flag}(\mathbb{H}^n)$ (see Theorem 4 in \cite{HLS2}),
we obtain that
$$ A_1 \leq C\Big(\|T\|_{2,2}+\|T\|_{\mathcal{M}_{flag}^{M^{\prime }+\delta },\mathcal{M}_{flag}^{M^{\prime }+\delta }}\Big)^p\|f\|^p_{H^p_{flag}(\mathbb{H}^n)}. $$
Similarly we can estimate the terms $A_2, A_3$ and $A_4$. We leave the details to the reader.

Now by Proposition \ref{boundedness on molecular} we obtain that the claim \eqref{claim 3} holds, which implies that
$$\|R_N(f)\|_{H^p_{flag}(\mathbb{H}^n)}\leq  C2^{-N}\|f\|_{H^p_{flag}(\mathbb{H}^n)}.$$ Thus, choosing $N$ large enough implies that
$T_N$ is invertible and $T_N^{-1}$ is bounded on $H^p_{flag}(\mathbb{H}^n)$.
Set $h=T_{\alpha,N}^{-1}f$. Then
\begin{eqnarray*}
f(x,y)&=& T_{\alpha,N} (T_{\alpha,N}^{-1}f)  \\
&=&\sum_{j,k\in \mathbb{Z}}\sum_{\substack{R=I\times J,\\\ell(I)=2^{-j-N},\\\ell(J)=2^{-j-N}+2^{-k-N} }}|R| \widetilde{\psi}_{j,k}((x,y)\circ(x_I,y_J)^{-1}) (\psi_{j,k}*h)(x_I,y_J).
\end{eqnarray*}
\end{proof}

We now return to Theorem \ref{main theorem}.
\begin{proof}[Proof of Theorem \ref{main theorem}]
We first verify \eqref{flag}. To this end,
applying  the discrete version of the reproducing
formula \eqref{eq:1} for $f$ in the term $\psi^{\prime} _{\mathcal{R}}\ast T(f)\left( z_{
\mathcal{R}},u_{\mathcal{R}}\right)$ given in \eqref{flag} implies that
\begin{eqnarray*}
&&\psi^{\prime} _{\mathcal{R}}\ast T(f)\left( z_{
\mathcal{R}},u_{\mathcal{R}}\right)\\
&=&\psi^{\prime}_{\mathcal{R}}\ast K\\
&&\ast \bigg(\sum_{j',k'\in \mathbb{Z}}\sum_{\substack{R'=I'\times J',\\ \ell(I')=2^{-j'-N},\\\ell(J')=2^{-j'-N}+2^{-k'-N} }}|R'| \widetilde{\psi}_{j',k'}((x,y)\circ(x_{I'},y_{J'})^{-1}) (\psi_{j',k'}*h)(x_{I'},y_{J'})\bigg)\left( z_{
\mathcal{R}},u_{\mathcal{R}}\right) \\
&=&\sum_{j',k'\in \mathbb{Z}}\sum_{\substack{R'=I'\times J',\\ \ell(I')=2^{-j'-N},\\\ell(J')=2^{-j'-N}+2^{-k'-N} }}|R'|
\psi^{\prime}_{\mathcal{R}}\ast K
\ast \widetilde{\psi}_{j',k'}((z_{
\mathcal{R}},u_{\mathcal{R}})\circ(x_{I'},y_{J'})^{-1}) (\psi_{j',k'}*h)(x_{I'},y_{J'}).
\end{eqnarray*}

Then, by Lemma \ref{lemma orth} to the term $\psi^{\prime}_{\mathcal{R}}\ast K
\ast \widetilde{\psi}_{j',k'}((z_{
\mathcal{R}},u_{\mathcal{R}})\circ(x_{I'},y_{J'})^{-1})$ in the right-hand side of the last equality above, we obtain that
\begin{eqnarray*}
&&|\psi^{\prime}_{\mathcal{R}}\ast T(f)\left( z_{
\mathcal{R}},u_{\mathcal{R}}\right)|\\
&\leq&\sum_{j',k'\in \mathbb{Z}}2^{-|j-j'|\varepsilon}2^{-|k-k'|}\sum_{\substack{R'=I'\times J',\\ \ell(I')=2^{-j'-N},\\\ell(J)=2^{-j'-N}+2^{-k'-N} }}|R'|
\frac{2^{-(j \wedge j')/2}}{(2^{-(j \wedge j')}
+ |z_{\mathcal{R}}-x_{I'}|)^{2n + \frac{1}{2}}}\\
&&\times
 \frac{2^{-(k \wedge k')/4}}{(2^{-k\wedge k'} + |u_{\mathcal{R}}-y_{J'}|)^{1 + \frac{1}{4}}}|(\psi_{j',k'}*h)(x_{I'},y_{J'})|\hskip.5cm \mbox{if  } 2(j \wedge j') \geq k \wedge k',
\end{eqnarray*}
and
\begin{eqnarray*}
&&|\psi^{\prime} _{\mathcal{R}}\ast T(f)\left( z_{
\mathcal{R}},u_{\mathcal{R}}\right)|\\
&\leq&\sum_{j',k'\in \mathbb{Z}}2^{-|j-j'|\varepsilon}2^{-|k-k'|}\sum_{\substack{R'=I'\times J',\\ \ell(I')=2^{-j'-N},\\\ell(J)=2^{-j'-N}+2^{-k'-N} }}|R'|
\frac{2^{-(j \wedge j')/2}}{(2^{-(j \wedge j')}
+ |z_{\mathcal{R}}-x_{I'}|)^{2n + \frac{1}{2}}} \\
&&\times
\frac{2^{-(j \wedge j')/2}}{(2^{-(j \wedge j')}
+ \sqrt{|u_{\mathcal{R}}-y_{J'}|})^{2  + \frac{1}{2}}}|(\psi_{j',k'}*h)(x_{I'},y_{J'})|\hskip.5cm \mbox{if  } 2(j \wedge j') < k \wedge k'.
\end{eqnarray*}

Using Lemma 7 in \cite{HLS, HLS2}, for $ {{4n}\over {4n+1}}<r<p$ and any $\left( z_{
\mathcal{R}}^*,u_{\mathcal{R}}^*\right)\in R$, we get that
\begin{eqnarray*}
&&|\psi^{\prime} _{\mathcal{R}}\ast T(f)\left( z_{
\mathcal{R}},u_{\mathcal{R}}\right)|\\
&\leq& C\sum_{j',k'\in \mathbb{Z}}2^{-|j-j'|\varepsilon}2^{-|k-k'|} 2^{({1\over r}-1)N(2n+1)}
2^{[2n(j\wedge j' -j')+(k\wedge k'-k')](1-{1\over r})}   \\
&&\times \Bigg( \mathcal{M}_s\bigg[\bigg(\sum_{\substack{R'=I'\times J',\\ \ell(I')=2^{-j'-N},\\\ell(J)=2^{-j'-N}+2^{-k'-N} }}|(\psi_{j',k'}*h)(x_{I'},y_{J'})|\chi_{I'}\chi_{J'}\bigg)^r\bigg]\Bigg)^{1\over r}\left( z_{
\mathcal{R}}^*,u_{\mathcal{R}}^*\right)\\
&&+ C\sum_{j',k'\in \mathbb{Z}:\ 2(j \wedge j') < k \wedge k' }2^{-|j-j'|\varepsilon}2^{-|k-k'|} 2^{({1\over r}-1)N(2n+1)}
2^{[2n(j\wedge j' -j')+(j\wedge j'-j'\wedge k')](1-{1\over r})}\\
&&\quad\times \Bigg( \mathcal{M}\bigg[\bigg(\sum_{\substack{R'=I'\times J',\\ \ell(I')=2^{-j'-N},\\\ell(J)=2^{-j'-N}+2^{-k'-N} }}|(\psi_{j',k'}*h)(x_{I'},y_{J'})|\chi_{I'}\chi_{J'}\bigg)^r\bigg]\Bigg)^{1\over r}\left( z_{
\mathcal{R}}^*,u_{\mathcal{R}}^*\right),\\
\end{eqnarray*}
where $\mathcal {M}$ is the Hardy-Littlewood maximal function and $\mathcal{M}_s$ is the strong maximal function on $\mathbb{H}^n,$ respectively.

Applying H\"older's inequality and Fefferman-Stein vector valued maximal inequality and summing over $\mathcal{R}\in
\mathsf{R}_{vert}$ yield
\begin{eqnarray*}
&&\bigg\|\bigg\{\sum_{\mathcal{R}\in
\mathsf{R}_{vert}}\left\vert \psi^{\prime} _{\mathcal{R}}\ast T(f)\left( z_{%
\mathcal{R}},u_{\mathcal{R}}\right) \right\vert ^{2}\chi _{\mathcal{R}%
}\left( z,u\right) \bigg\} ^{{\frac{{1}}{{2}}}}\bigg\|_p\\
&&\leq C\bigg\|\bigg\{\sum_{\mathcal{R}\in
\mathsf{R}_{vert}}\Bigg\vert \sum_{j',k'\in \mathbb{Z} }2^{-|j-j'|\varepsilon}2^{-|k-k'|}2^{[2n(j\wedge j' -j')+(k\wedge k'-k')](1-{1\over r})}\\
&&\hskip.5cm\Bigg( \mathcal{M}_s\bigg[\bigg(\sum_{\substack{R'=I'\times J',\\ \ell(I')=2^{-j'-N},\\\ell(J)=2^{-j'-N}+2^{-k'-N} }}|(\psi_{j',k'}*h)(x_{I'},y_{J'})|\chi_{I'}\chi_{J'}\bigg)^r\bigg]\Bigg)^{1\over r}\left( z_{
\mathcal{R}}^*,u_{\mathcal{R}}^*\right) \Bigg\vert ^{2}\chi _{\mathcal{R}%
}\left( z,u\right) \bigg\} ^{{\frac{{1}}{{2}}}}\bigg\|_p\\
&&\leq C\bigg\|\bigg\{ \sum_{j',k'\in \mathbb{Z} }\sum_{\substack{R'=I'\times J',\\ \ell(I')=2^{-j'-N},\\\ell(J)=2^{-j'-N}+2^{-k'-N} }}|(\psi_{j',k'}*h)(x_{I'},y_{J'})|^{2}\chi_{I'}(\cdot)\chi_{J'}(\cdot)\bigg\}^{{\frac{{1}}{{2}}}}\bigg\|_p\\
&&\leq C\|h\|_{H^p(\mathbb{H}^n)}\\
&&\leq C\|f\|_{H^p(\mathbb{H}^n)}.
\end{eqnarray*}%

The proof for \eqref{one-parameter} is similar and easier. The proof of Theorem \ref{main theorem} is concluded.
\end{proof}

\bigskip
{\bf Acknowledgement:} J. Li is supported by ARC DP 160100153.

\bibliographystyle{amsplain}

\end{document}